\documentclass[11pt]{article} 
\usepackage[utf8]{inputenc} 

\usepackage[margin=1in]{geometry} 
\usepackage{graphicx} 

\usepackage{dsfont} 

\usepackage{booktabs} 
\usepackage{array} 
\usepackage{paralist} 
\usepackage{verbatim} 
\usepackage{mathrsfs}
\usepackage{amssymb}
\usepackage{amsthm}
\usepackage{amsmath,amsfonts,amssymb}
\usepackage{esint}
\usepackage{graphics}
\usepackage{enumerate}
\usepackage{mathtools}
\usepackage{xfrac}
\usepackage{subcaption}
\usepackage{stmaryrd}
 \usepackage{mathabx}

\usepackage[dvipsnames]{xcolor}
\usepackage[colorlinks=true, pdfstartview=FitV, linkcolor=blue, citecolor=blue, urlcolor=blue]{hyperref}
\usepackage[normalem]{ulem}

\usepackage{tikz}
\usetikzlibrary{calc}
\usepackage{pgf}
\usetikzlibrary{external}
\numberwithin{equation}{section}
\numberwithin{figure}{section}

\newtheorem{theorem}{Theorem}[section]

\newtheorem{corollary}[theorem]{Corollary}
\newtheorem{proposition}[theorem]{Proposition}

\theoremstyle{definition}
\newtheorem{definition}[theorem]{Definition}

\newtheorem{remark}[theorem]{Remark}

\newcommand*{\Z}{\ensuremath{\mathbb{Z}}}

\newcommand*{\R}{\ensuremath{\mathbb{R}}}

\newcommand{\eps}{\varepsilon}

\newcommand{\ep}{\eps}

\DeclareSymbolFont{boldoperators}{OT1}{cmr}{bx}{n}
\SetSymbolFont{boldoperators}{bold}{OT1}{cmr}{bx}{n}
\usepackage{accents}

\newcommand{\T}{\mathbb{T}}

\newcommand{\dif}{\mathrm{d}}

\newcommand{\Law}{\mathrm{Law}}

\def\XXint#1#2#3{{\setbox0=\hbox{$#1{#2#3}{\int}$}
\vcenter{\hbox{$#2#3$}}\kern-.5\wd0}}

\let\originalleft\left
\let\originalright\right
\renewcommand{\left}{\mathopen{}\mathclose\bgroup\originalleft}
\renewcommand{\right}{\aftergroup\egroup\originalright}

\newcommand{\indc}{\mathds{1}}

\newcommand{\E}{\mathbb{E}}

\renewcommand{\hat}{\widehat}

\renewcommand{\phi}{\varphi}

\makeatletter
\pgfmathdeclarefunction{erf}{1}{%
  \begingroup
    \pgfmathparse{#1 > 0 ? 1 : -1}%
    \edef\sign{\pgfmathresult}%
    \pgfmathparse{abs(#1)}%
    \edef\x{\pgfmathresult}%
    \pgfmathparse{1/(1+0.3275911*\x)}%
    \edef\t{\pgfmathresult}%
    \pgfmathparse{%
      1 - (((((1.061405429*\t -1.453152027)*\t) + 1.421413741)*\t
      -0.284496736)*\t + 0.254829592)*\t*exp(-(\x*\x))}%
    \edef\y{\pgfmathresult}%
    \pgfmathparse{(\sign)*\y}%
    \pgfmath@smuggleone\pgfmathresult%
  \endgroup
}
\makeatother

\usepackage{titlesec}

\newcommand{\addperiod}[1]{#1.}
\titleformat{\section}
   {\centering\normalfont\Large}{\thesection.}{0.5em}{}
\titleformat*{\subsection}{\bfseries}
\titleformat{\subsubsection}[runin]
  {\normalfont\bfseries}
  {\thesubsubsection.}
  {0.5em}
  {\addperiod}
\titleformat*{\subsubsection}{\normalfont\itshape}
\titleformat*{\paragraph}{\bfseries}
\titleformat*{\subparagraph}{\large\bfseries}

\title{Fourier mass lower bounds for Batchelor-regime passive scalars}

\author{William Cooperman\thanks{ETH Z\"urich.
{\footnotesize \href{mailto:bill@cprmn.org}{bill@cprmn.org}.}
}
\and
Keefer Rowan\thanks{Courant Institute of Mathematical Sciences,  New York University.
{\footnotesize \href{mailto:keefer.rowan@cims.nyu.edu}{keefer.rowan@cims.nyu.edu}.}
}
}
\date{\today}

\usepackage[nottoc,notlot,notlof]{tocbibind}

\begin{document}

\maketitle

\begin{abstract}
    Batchelor predicted that a passive scalar $\psi^\nu$ with diffusivity $\nu$, advected by a smooth fluid velocity, should typically have Fourier mass distributed as $|\hat \psi^\nu|^2(k) \approx |k|^{-d}$ for $|k| \ll \nu^{-1/2}$. For a broad class of velocity fields, we give a quantitative lower bound for a version of this prediction summed over constant width annuli in Fourier space. This improves on previously known results, which require the prediction to be summed over the whole ball.
\end{abstract}


\section{Introduction}

In this paper, we study the Fourier mass distribution of solutions $\psi^\nu_t$ to the stochastically forced advection-diffusion equation
\begin{equation}
  \label{eq:randomly-forced-transport}
  \begin{cases}
    \partial_t \psi^\nu_t + u_t \cdot \nabla \psi^\nu_t - \nu \Delta \psi^\nu_t = g \dif W_t \qquad \text{for } t \geq 0,\\
    \psi^\nu_0 = 0,
    \end{cases}
\end{equation}
where $u_t$ is a random, exponentially mixing, divergence-free, sufficiently regular velocity field (e.g.\ a solution to a stochastically forced fluid equation), $W_t$ is a standard Brownian motion independent of $u_t$, and $g \in L^2$. This setting is sometimes referred to as \textit{Batchelor-regime passive scalar turbulence}.

In 1959, Batchelor~\cite{batchelor_small-scale_1959} predicted that, if $u$ is sufficiently regular and $k \in \Z^d$ is a frequency below the dissipative scale $|k| \ll \nu^{-1/2}$, then the Fourier mass at frequency $k$ is approximately
\begin{equation}
  \label{eq:batchelor-original-prediction}
  \E|\hat{\psi}^\nu_t(k)|^2 \approx |k|^{-d}
\end{equation}
for $t > 0$ large.

Although there have been many experimental works~\cite{GibsonC.H.1963Tues, GrantH.L.1968Tsot, Nye_Brodkey_1967, Dillon80, PhysRevLett.85.3636,PhysRevLett.93.214504} studying the conditions under which Batchelor's prediction holds, the first mathematically rigorous result in this direction was given recently by Bedrossian, Blumenthal, and Punshon-Smith~\cite{bedrossian_batchelor_2021}, who proved that a \textit{cumulative} version of Batchelor's prediction holds under some general assumptions on the mixing properties and regularity of the velocity $u_t$. More precisely, they proved the cumulative estimate
\begin{equation}
  \label{eq:cumulative-prediction}
  \lim_{t \to\infty}\E \sum_{|k| \leq r} |\hat{\psi}^\nu_t(k)|^2 \approx \log r
\end{equation}
for each $0 \ll r \ll \nu^{-1/2}$, which is consistent with Batchelor's pointwise prediction~\eqref{eq:batchelor-original-prediction}.

On the other hand, recent work of Blumenthal and Huynh~\cite{blumenthal_sparsity_2024} shows that, without an additional assumption relating to isotropy of the velocity, it is at least necessary to sum Batchelor's prediction over all frequencies $k$ of a given modulus. Indeed,~\cite{blumenthal_sparsity_2024} produces a discrete-time system with mass approximately localized to a cone in Fourier space, causing the lower bound of~\eqref{eq:batchelor-original-prediction} to fail dramatically.

In the present paper, we prove a version of Batchelor's prediction which only sums over frequencies in a constant-width annulus. In Theorem~\ref{thm:general}, we prove quantitatively that for any small $\theta > 0$, there are large constants $h, K > 0$ such that for all $\nu$ sufficiently small, the bound
\begin{equation}
  \label{eq:our-bound}
  K^{-1}hr^{-1} \leq \lim_{t \to\infty}\E \sum_{r \leq |k| \leq r+h} |\hat{\psi}^\nu_t(k)|^2 \leq Khr^{-1}
\end{equation}
holds for at least $(1-\theta)$-fraction of radii $r \in [1, \nu^{-1/2}]$, measured with respect to the normalized logarithmic density~\eqref{eq:log-density-measure-def}.

When compared to the cumulative bound~\eqref{eq:cumulative-prediction}, the constant-width annulus bound~\eqref{eq:our-bound} has two main benefits. First, the constant-width bound identifies (up to a sub-algebraic factor in $\nu^{-1}$) the Batchelor scale $\nu^{-1/2}$; that is, the same bound on a smaller range $[0, \nu^{-1/2 + \varepsilon}]$ of radii would be strictly weaker, since the smaller range misses a constant fraction of radii (measured by logarithmic density). Second, the constant-width bound shows that the Fourier mass is not distributed in sparse chunks. Indeed, a distribution of mass $2^n$ at wavenumbers $(2^{2^n}, 0)$, and no mass elsewhere, would be consistent with the cumulative bound~\eqref{eq:cumulative-prediction} but violate~\eqref{eq:our-bound}. Finally, we can recover the cumulative lower bound of~\eqref{eq:cumulative-prediction} by integrating the constant-width bound~\eqref{eq:our-bound} over any positive fraction of radii.

We also prove, in~\eqref{eq:general-as} and Theorem~\ref{thm:compact-fourier-support}, that stronger conclusions are possible under stronger regularity or stochastic integrability assumptions on $u_t$. In particular, under the strongest assumption that $u_t$ is deterministically bounded with compact Fourier support (as in Pierrehumbert's alternating shear example), we prove that the lower bound in~\eqref{eq:our-bound} holds for every $r \in [0, \nu^{-1/2}]$.

All of our conclusions follow from a novel \textit{flux inequality}, stated as Theorem~\ref{thm:flux}.

\subsection{Setting and assumptions}

For each $t \geq 0$, let $u_t \colon \T^d \to \R^d$ be a random divergence-free velocity with $u \in L^1_{\mathrm{loc},t}W^{1,\infty}_x$ almost surely. Assume that the $u_t$ process is stationary in time, i.e. $\Law(t \mapsto u_t) = \Law(t \mapsto u_{t+s})$ for all $s \geq 0$. Let $\nu > 0$ and $g \in L^2(\T^d)$ with $\int g\,\dif x =0$ and $\|g\|_{L^2} = 1$. Define
\[N_0 := \inf \{r \geq 1 : \|\Pi_{\geq r} g\|_{L^2}^2 \leq \tfrac{1}{4}\},\]
where $\Pi_{\geq r}$ is the Fourier projector defined in Definition~\ref{def:fourier-projector}.

Our main assumption is that $u_t$ exponentially mixes the initial data $g$ uniformly in $\nu$. That is, for $\phi^\nu_t$ defined to be the (random) solution to the equation
\begin{equation}
    \label{eq:phi-equation}
    \begin{cases}
    \partial_t \phi^\nu_t + u_t \cdot \nabla \phi^\nu_t - \nu \Delta \phi^\nu_t =0 \qquad \text{for } t \geq 0,\\
    \phi^\nu_0 = g,
    \end{cases}
\end{equation}
we suppose there exist $K \geq 1, \gamma>0$, such that for all $\nu,t>0$,
\begin{equation}
    \label{eq:u-mixes}
    \E \|\phi_t^\nu\|_{H^{-1}}^2 \leq K e^{-\gamma t}.
\end{equation}
Throughout, we denote by $C>0$ an arbitrary constant depending on $K,\gamma,N_0,d$.

\subsection{Main results}

In this paper, we focus primarily on proving a version of the lower bound of~\eqref{eq:our-bound}. In particular, we show that the radii $r$ such that $\lim_{t \to\infty} \sum_{r \leq |k| \leq r+h} \E  |\hat \psi^\nu_t(k)|^2  \ll \frac{h}{r}$ are rare. By rare, we mean the set of $r$ where the ratio of these terms is small has a small density. Due to the dynamics of the advective term in of~\eqref{eq:randomly-forced-transport}, which typically allow Fourier mass to move from wavenumbers of magnitude $|k|$ to wavenumbers of magnitude $2|k|$ in unit time, the natural density to control is not the uniform density on $r \in [1,\infty)$, but rather the logarithmic density, defined in~\eqref{eq:log-density-measure-def}. Our first and most general result shows that if the velocity $u_0$ is sufficiently regular, then the density of radii that violate the lower bound of~\eqref{eq:our-bound} tends to $0$ quantitatively as $h,K \to \infty$, uniformly in $\nu$.

Define the ``bad'' set of undercharged radii, at width $h>0$ and lower bound $\alpha>0$, by
\begin{equation}
    \label{eq:bad-radii-def}
B^\nu_{h,\alpha} := \Big\{r \in [1,\infty) :\lim_{t \to\infty} \sum_{r \leq |k| \leq r+h} \E  |\hat \psi^\nu_t(k)|^2\leq \frac{\alpha h}{r}\Big\},
\end{equation}
and define, for each $0< a < b < \infty$, the probability measure $\mu_{a,b}$ on $[a,b]$, given for $E \subseteq \R$ by
\begin{equation}
    \label{eq:log-density-measure-def}
    \mu_{a,b}(E) := \frac{1}{\log(b/a)}\int_{E \cap [a,b]} \frac{1}{r}\,\dif r.
\end{equation}

Our first result states that, if $h$ is large and $\alpha$ is small, then $B^{\nu}_{h,\alpha}$ has small $\mu_{N_0,R}$-measure for $R$ up to the Batchelor scale $\nu^{-1/2}$.

\begin{theorem}
\label{thm:general}
    For each $h>0$ and
    \begin{equation}
      \label{eq:valid-R-range}
      C\lor h \leq R \leq\frac{\nu^{-1/2}}{C (\log \nu^{-1})^{1/2}},
    \end{equation}
    for each $p,q\geq 2$, the logarithmic density of bad annuli is bounded by
        \begin{equation}
            \label{eq:general-general}
        \mu_{N_0,R}(B^\nu_{h,\alpha}) \leq C^q h^{2 - q} \alpha^{-1} + C^{p+q} \E \|u_0\|_{C^{q+(d+1)/2}}^p \alpha^{p-1}.
        \end{equation}
    In particular, if we have the a.s.\ bound
    \[\|u_0\|_{C^{q+(d+1)/2}} \leq D,\]
    then taking $p \to \infty$ shows
    \begin{equation}
        \label{eq:general-as}
    \mu_{N_0,R}(B^\nu_{h, C^{-1} D^{-1}}) \leq C^q h^{2-q} D.
    \end{equation}
\end{theorem}

\begin{remark}
As mentioned in the introduction, the range~\eqref{eq:valid-R-range} of $R$ for which Theorem~\ref{thm:general} holds corresponds (up to a $\log$) to the expected range of frequencies in Batchelor's prediction. Indeed, the same conclusions for a smaller range $C \lor h \leq R \leq \nu^{-1/2+\varepsilon}$ for $\varepsilon > 0$ are strictly weaker, since the smaller range misses a constant fraction of radii, i.e.\ $\mu_{N_0, R_{\text{high}}}([\nu^{-1/2+\varepsilon}, R_\text{high}]) \geq 2\varepsilon + o_{\nu \to 0^+}(1)$ for $R_\text{high} := C^{-1}(\nu \log \nu^{-1})^{-1/2}$.
\end{remark}
\begin{remark}
  In the interest of simplicity, we stated a suboptimal bound with respect to the regularity of $u_0$. Indeed, if we only assumed bounds on $\E\|u_0\|_{C^{q+d/2+\varepsilon}}^p$ for some $\varepsilon > 0$ (or even $\E\|u_0\|_{C^{q+d/2}\log^{1+\varepsilon}C}^p$, etc.) then the same conclusion~\eqref{eq:general-general} would follow with an extra factor of $\varepsilon^{-p/2}$ in the second term; we only use the $C^{q+(d+1)/2}$ norm of $u_0$ to estimate $\||k|^q\hat{u}_0\|_{\ell^1}$ in~\eqref{eq:only-use-of-Cq}.
\end{remark}
\begin{remark}
  Although Theorem~\ref{thm:general} only states that the set $B^\nu_{h,\alpha}$ (where the \textit{lower} bound of the prediction~\eqref{eq:our-bound}is violated) is sparse, we observe that the corresponding statement for the \textit{upper} bound follows directly from the cumulative bound~\eqref{eq:cumulative-prediction} from~\cite{bedrossian_batchelor_2021}. Indeed, since Fourier mass is nonnegative, Markov's inequality shows that a cumulative upper bound implies a corresponding pointwise upper bound for most (measured in the sense of logarithmic density) frequencies---as is explained in Subsection~\ref{ss:upper-bounds}. On the other hand, the cumulative \textit{lower} bound itself says nothing about pointwise lower bounds; it is consistent with the cumulative bounds that the Fourier mass is supported on a set of frequencies with arbitrarily small logarithmic density.
\end{remark}

Some velocity fields satisfying the assumptions of~\eqref{eq:general-general} (but not those of~(\ref{eq:general-as}, \ref{eq:compact-general}, \ref{eq:compact-as})) are constructed in~\cite{bedrossian_almost-sure_2021} as solutions to stochastically forced fluid equations. In particular, one can take $u_t$ to be the solution to a stochastically forced 2D incompressible Navier--Stokes equation.

Our second result is that, if $u_t$ is compactly supported in Fourier space, then we do not need to send $h \to \infty$, but rather only need to take $h$ sufficiently large.

\begin{theorem}
\label{thm:compact-fourier-support}
    Suppose that for some deterministic $L>0$ and a random magnitude $X \geq 1$, we have that
    \[|\hat u_0(k)| \leq X\indc_{|k| \leq L}.\]
    Then there is $C > 0$, depending on $L$, such that for all
        \[C \leq R \leq\frac{\nu^{-1/2}}{C (\log \nu^{-1})^{1/2}},\]
    and $p \geq 2,$ we have the density estimate
    \begin{equation}
        \label{eq:compact-general}
    \mu_{C,R}(B^\nu_{2L,\alpha}) \leq C^p \E X^p \alpha^{p-1}.
    \end{equation}
    In particular, if we have the a.s.\ bound
    \[X \leq D,\]
    then
    \begin{equation}
        \label{eq:compact-as}
    \mu_{C,R}(B^\nu_{2L,(CD)^{-2}}) =0.
    \end{equation}
\end{theorem}

Some velocity fields satisfying the assumptions of~\eqref{eq:compact-as} are constructed in~\cite{cooperman_harris_2024,blumenthal_exponential_2023,navarro-fernandez_exponential_2025}. In particular, one can take $u_t$ to be the alternating sine shear mixing example of Pierrehumbert~\cite{pierrehumbert_tracer_1994}. We further note that~\cite{bedrossian_almost-sure_2021} shows that Stokes flow with finitely many stochastically forced Fourier modes satisfies the assumptions of~\eqref{eq:compact-general}, with $\E X^p <\infty$ for all $p$.

In order to prove Theorems~\ref{thm:general} and~\ref{thm:compact-fourier-support}, it is useful to introduce the the \textit{asymptotic mass measure}~$m^\nu$, a positive measure on radii $r \in [0,\infty)$, defined so that for $E\subseteq [0,\infty)$,
\begin{equation}
    \label{eq:mass-measure-def}
    m^\nu(E) := \sum_{|k| \in E} \lim_{t \to\infty}\E |\hat \psi^\nu_t(k)|^2.
\end{equation}
Under this definition, we have the equivalent characterization
\[B^\nu_{h,\alpha} = \big\{r \in [1,\infty) : m^\nu([r,r+h])\leq \frac{\alpha h}{r}\big\}.\]
We will also use the \textit{random instantaneous mass measure} associated to the $\phi_t$ process, defined for $E \subseteq [0,\infty)$ by
\begin{equation}
\label{eq:instantaneous-mass-measure-def}
m_{t}^\nu(E) := \sum_{|k| \in E} |\hat \phi^\nu_t(k)|^2.
\end{equation}
The It\^o isometry implies
\[m^\nu(E) =\E \int_0^\infty m^\nu_t\,\dif t.\]

Our last result is a general flux relation that is the primary tool we use to prove Theorems~\ref{thm:general} and~\ref{thm:flux}.

\begin{theorem}
    \label{thm:flux}
    For any increasing weight $w : [0,\infty) \to [1,\infty]$, for all $\nu,t,r,$ as well as all random events $\omega,$ we have the pointwise inequality,
    \[\frac{d}{dt} \|\Pi_{\geq r} \phi_t^\nu\|_{L^2}^2 + 2 \nu  \|\Pi_{\geq r} \nabla \phi_t^\nu\|_{L^2}^2   \leq 4\pi r \|w(|k|) \hat u_t\|_{\ell^1} \, w^{-1}*m^\nu_t(r),\]
    where
    \begin{equation}
      \label{eq:conv-def}
      w^{-1} * m_{t,\omega}^\nu(r) := \int_0^\infty \frac{1}{w(|r-z|)} m_{t,\omega}^\nu(\dif z).
    \end{equation}
  \end{theorem}

\subsection{Acknowledgments}

We thank Vlad Vicol for helpful feedback on an early draft. W.C.\ was partially supported by NSF grant DMS-2303355. K.R.\ was partially supported by the NSF Collaborative Research Grant DMS-2307681 and the Simons Foundation through the Simons Investigators program.

\section{Discussion}

\label{sec:discussion}

\subsection{Background and previous results}

\subsubsection{Physical background}
By arguing that a smooth fluid velocity should be locally well  approximated by a linear hyperbolic flow, Batchelor~\cite{batchelor_small-scale_1959} predicted that for a forced passive scalar with $\nu$ diffusivity, one should typically have
\begin{equation}
\label{eq:modewise-batchelor}
\text{if } 1\ll  |k| \ll \nu^{-1/2}, \quad |\hat \psi^\nu(k)|^2 \approx |k|^{-d}.
\end{equation}
See~\cite[Section 1]{bedrossian_batchelor_2021} for a thorough discussion of historical and physical context.

The key motivation for studying Batchelor's prediction is that a passive scalar exhibits the most mathematically accessible of the advective cascades that are central to the study of fluid turbulence. The most famous such cascade is the turbulence cascade of an infinite Reynolds number 3D fluid, described phenomenologically by Kolmogorov's K41 theory~\cite{kolmogorov_degeneration_1941,kolmogorov_dissipation_1941,kolmogorov_local_1941}; see~\cite[Chapter 6]{frisch_turbulence_1995} for a lucid exposition. K41 theory predicts that for a fluid velocity $v$ in the infinite Reynold's number limit, one typically has
\[|k|^2 |\hat v|^2(k) \approx |k|^{-5/3};\]
this is sometimes referred to as the ``five-thirds law'' (which is something of a misnomer, as the ``five-thirds law'' is apparently false, due to intermittency corrections~\cite[Chapter 8]{frisch_turbulence_1995}). Obukhov~\cite{obukhov_structure_1949} and Corrsin~\cite{corrsin_spectrum_1951} make analogous predictions on the Fourier spectra of passive scalars advected by rough ``turbulent'' velocity fields.

All of these spectral predictions share the feature that a certain distribution of mass over the length scales (or equivalently Fourier wavenumbers) is expected due to the action of the advection. This cascading creation of length scales with a precise and universal distribution of mass is one of the central phenomena of turbulence.

The cases of a turbulent fluid and a passive scalar advected by a rough velocity field are considerably more difficult than the Batchelor setting. The phenomena of anomalous dissipation and intermittency make the situation exceedingly complex. Manifestations of these phenomena in the fluid equations are apparent in the convex integration solutions to the fluid equations and Onsager's conjecture; see for example~\cite{constantin_onsagers_1994,de_lellis_h-principle_2012,de_lellis_dissipative_2013,buckmaster_anomalous_2015,isett_proof_2018,novack_intermittent_2023} and in particular the reviews~\cite{buckmaster_convex_2019,de_lellis_weak_2022}. In the setting of a passive scalar with rough advecting flow, one has the freedom to choose the advecting velocity field, and so the phenomenon of anomalous dissipation can be rigorously exhibited as is shown in~\cite{drivas_anomalous_2022,colombo_anomalous_2023,armstrong_anomalous_2025,burczak_anomalous_2023,elgindi_norm_2024,rowan_anomalous_2024,johansson_anomalous_2024,hess-childs_universal_2025} among others. However, obtaining precise control of the distribution of Fourier mass in equilibrium is still beyond current analysis.

The Batchelor regime---in which a smooth velocity field advects a passive scalar with small diffusivity---has neither anomalous dissipation nor (seemingly) intermittency and is therefore much more amenable to analysis. The primary phenomenon of interest in this setting is (exponential) mixing, which is the mechanism that underlies the cascade of mass to small length scales. Mixing is, at this point, a well studied phenomenon, as is made clear below.

\subsubsection{Remarks on the model}

We model our fluid velocity field as a spatially smooth, time stationary random process. This is consistent with taking the fluid velocity to be the solution to a stochastically forced fluid equation with fixed positive viscosity. Taking the fluid velocity to have some randomness is a straightforward way to access ``generic'' behavior of fluid velocities. As discussed below, it is also only in this random setting that we have results guaranteeing uniform-in-diffusivity exponential mixing, which we assume in~\eqref{eq:u-mixes}.

Besides the random fluid velocity, we also take our passive scalar equation~\eqref{eq:randomly-forced-transport} to have a temporal white noise forcing. This is a somewhat idealized model of a generic forcing, though it allows us to exploit the It\^o isometry to avoid having to consider subtle cancellations.

We then consider the mass distribution of the solution to~\eqref{eq:randomly-forced-transport} at long times. As in~\cite{bedrossian_batchelor_2021}, this is equivalent to studying the stationary measure for the process defined by~\eqref{eq:randomly-forced-transport} and as such gives information about the expected distribution of Fourier mass in the stationary measure. We refrain from introducing stationary measures to simplify the exposition.

\subsubsection{Previous results on mixing, enhanced dissipation, and the Batchelor spectrum}

Working in a similar setting,~\cite{bedrossian_batchelor_2021} studies a cumulative version of the Batchelor prediction, where Fourier modes are summed over a ball. In this case, the prediction becomes~\eqref{eq:cumulative-prediction}. This cumulative spectrum is clearly established in~\cite{bedrossian_batchelor_2021} as a consequence of a uniform-in-diffusivity exponential mixing estimate of the form~\eqref{eq:u-mixes}. More precisely, the \textit{upper bound} is a consequence of exponential mixing, while the \textit{lower bound} is a consequence of regularity of $u$, which guarantees that the Fourier mass takes $\approx \log R$ time to escape the ball of radius $R$ in Fourier space.

The takeaway from~\cite{bedrossian_batchelor_2021} is that uniform-in-diffusivity exponential mixing causes some form of the Batchelor spectrum. However, their analysis cannot give a more precise form than the cumulative bound: they do not rule out that large regions of Fourier space are devoid of mass.

The uniform-in-diffusivity exponential mixing needed to establish the cumulative Batchelor spectrum in~\cite{bedrossian_batchelor_2021} is provided by the series of papers~\cite{bedrossian_lagrangian_2022,bedrossian_almost-sure_2022,bedrossian_almost-sure_2021} in which the authors establish a uniform-in-diffusivity exponential mixing result for the solutions to a variety of stochastically forced fluid equations, most notably the 2D incompressible Navier-Stokes equation at fixed viscosity. Since~\cite{bedrossian_almost-sure_2021}, there have been additional constructions of uniform-in-diffusivity exponential mixers. Building on~\cite{blumenthal_exponential_2023},~\cite{cooperman_harris_2024} shows a uniform-in-diffusivity mixing result for a variety of velocity fields that are sampled independently on each unit time interval $[n,n+1],$ such as the Pierrehumbert mixing example~\cite{pierrehumbert_tracer_1994}. \cite{navarro-fernandez_exponential_2025} uses Villani’s hypocoercivity method~\cite{villani_hypocoercivity_2009} to prove uniform-in-diffusivity exponential mixing for the velocity field $u(x+ B_t)$ where $B_t$ is a standard Brownian motion and $u$ is the ``simple cellular flow''. \cite{coti_zelati_mixing_2024} shows uniform-in-diffusivity negative regularity exponential mixing for a white-in-time, correlated-in-space self-similar Gaussian field.

Uniform-in-diffusivity exponential mixing is somewhat harder to establish than exponential mixing for the $\nu=0$ case. However, the $\nu=0$ has seen more attention: while all examples of uniform-in-diffusivity exponential mixing are random, there are deterministic examples of exponential mixers for $\nu=0$. For a particular choice of initial data,~\cite{alberti_exponential_2019} provides a uniformly spatially Lipschitz velocity field that exponentially mixes. The works \cite{elgindi_optimal_2023,myers_hill_exponential_2022} provide uniformly spatially Lipschitz velocity fields that exponentially mix all (sufficiently regular) initial data. The authors of the current work also establish $\nu=0$ exponential mixing in the random case for the 2D Navier-Stokes equation with degenerate stochastic forcing~\cite{cooperman_exponential_2024}. It seems unlikely that there is a general result showing that $\nu=0$ exponential mixing implies uniform-in-diffusivity exponential mixing. However, we note that the ``negative regularity mixing'' of the $\nu=0$ problem established~\cite{bedrossian_negative_2024} implies uniform-in-diffusivity exponential mixing, though negative regularity mixing only appears to be possible in the setting of a random velocity field.

We note that even without exponential mixing, one is still able to deduce some form of a lower bound for the Batchelor spectrum. In particular, following the reasoning of Subsection~\ref{ss:special-case}, we derive the convolved lower bound on the mass measure $m^\nu$ given by~\eqref{eq:convolved-mass-bound}:
\[\text{if }\|w \hat u\|_{\ell^1} \leq A_w\; \text{a.s} \quad \text{then for all } N_0 \leq r \leq D_\nu,\, w^{-1}*m^\nu(r)\geq \frac{1}{16 \pi A_w r},\]
where $D_\nu$---defined by~\eqref{eq:dissipation-scale-def}---is the inverse of the length scale below which $\frac{1}{2}$ of the energy is dissipated. Thus one gets a version of the Batchelor spectrum lower bound (weak due to the ``averaging'' present from the $w^{-1} * m^\nu$ convolution) provided $D_\nu \to \infty$ as $\nu \to 0.$ One can readily verify that this is equivalent to the advecting flow $u$ being relaxation enhancing (also referred to as dissipation enhancing)~\cite{constantin_diffusion_2008}. As is shown in~\cite{feng_dissipation_2019,zelati_relation_2020}, relaxation enhancing is implied by mixing---with any mixing rate, not just exponentially fast mixing. Enhanced dissipation has been shown in an even wider variety of settings than exponential mixing, particularly in the case that there is a shear structure to the flow which precludes exponential mixing~\cite{bedrossian_enhanced_2017,feng_dissipation_2019,albritton_enhanced_2022,coti_zelati_enhanced_2023,villringer_enhanced_2024}.

The recent work~\cite{blumenthal_sparsity_2024} studies the Batchelor spectrum in the case of alternating applications of pure diffusion $e^{\nu \Delta}$ and a diffeomorphism of the torus which is a perturbation of Arnold's cat map, which is notably \textit{not} realized as the unit time map of a flow on $\T^2$. One result of~\cite{blumenthal_sparsity_2024} is that, in their setting, the version of the Batchelor spectrum given by~\eqref{eq:our-bound} is violated while the cumulative Batchelor spectrum prediction~\eqref{eq:cumulative-prediction} still holds. In particular,~\eqref{eq:our-bound} fails strongly for $r \leq \ep^{-\delta}$ where $\delta>0$ and $\ep$ is a measure of the size of the perturbation. The current paper shows that violations of this kind are not possible in the continuous time setting.

A second result of~\cite{blumenthal_sparsity_2024} shows that, in their setting, $\lim_{n \to \infty} \E |\hat \psi^\nu_n(k)|^2$ is essentially supported on $k$ in a cone in Fourier space, having very little mass on modes outside the cone. This leads to strong violations of the mode-wise Batchelor prediction~\eqref{eq:modewise-batchelor} for all radii. We expect that this phenomenon is also possible in continuous time, but we leave the construction of such an example to future work.

\subsection{Upper bounds on the mass distribution}
\label{ss:upper-bounds}

Our main results give lower bounds on the mass measure $m^\nu$ defined in~\eqref{eq:mass-measure-def}. It is then natural to ask about matching upper bounds. Additionally, we only give lower bounds for $|k| \ll \nu^{-1/2}$. For $|k| \gg \nu^{-1/2}$, the dynamics should be dominated by the $\nu \Delta$ dissipation, and so the lower bound should not generally hold.

\subsubsection{\texorpdfstring{The $|k| \ll \nu^{-1/2}$ range}{The low frequency range}}

First, when $|k| \ll \nu^{-1/2}$, we note the following upper bound. This bound is essentially identical to the cumulative Batchelor spectrum upper bound obtained in~\cite{bedrossian_batchelor_2021} and relies on the uniform-in-diffusivity exponential mixing.

\begin{proposition}
\label{prop:upper-density-bound}
    For all $R \geq 2$,
    \begin{equation}
    \label{eq:cumulative-upper-bound}
    m^\nu([1,R]) \leq C \log R.
    \end{equation}
\end{proposition}

\begin{proof}
    Using the exponential mixing, we note that for any $\alpha>0$
    \begin{align*}
        m^\nu([1,R]) &= \int_0^\infty \E \|\Pi_{\leq R} \phi_t^\nu\|_{L^2}^2
        \\&\leq \int_0^{\alpha \log R} \E \|\phi_t^\nu\|_{L^2}^2 + C R^2 \int_{\alpha \log R}^\infty e^{-C^{-1} t}\,\dif t
        \\&\leq \alpha \log R + C R^2 e^{-C^{-1} \alpha \log R}.
    \end{align*}
    Choosing $\alpha$ large enough, we conclude.
\end{proof}

One can then use Chebyshev's inequality to obtain a log-density bound on the ``overcharged'' radii. We define
\begin{equation}
    \overline{B}^\nu_{h,\alpha} := \big\{ r \in [1,\infty) : m^\nu([r, r+h]) \geq \frac{\alpha h}{ r}\big\}
\end{equation}
and compute
    \[
        \mu_{1,R}(\overline{B}^\nu_{h,\alpha}) = \frac{1}{\log R}\int_{\overline{B}^\nu_{h,\alpha} \cap [1,R]} \frac{1}{r}\,\dif r \leq \frac{1}{\alpha h \log R} \int_1^R m^\nu([r,r+h])\,\dif r \leq \frac{1}{\alpha \log R} m^\nu([1,R+h]) \leq C\alpha^{-1}.
    \]
We have proven the following density estimate.
\begin{proposition}
    For all $R \geq 2$,
    \[\mu_{1,R}(\overline{B}^\nu_{h,\alpha}) \leq C\alpha^{-1}.\]
\end{proposition}

We note that, in some sense, this upper bound matches the second term of the lower bound density estimate in~\eqref{eq:general-general} or the lower bound estimate in~\eqref{eq:compact-general}. It however does not match the estimates~\eqref{eq:general-as} or~\eqref{eq:compact-as}, which do not require taking a limit in $\alpha$.

\subsubsection{\texorpdfstring{The $|k| \gg \nu^{-1/2}$ range}{The |k| >> nu{-1/2]} range}}

The bound given by Proposition~\ref{prop:upper-density-bound} holds for all $R \geq 2$, however one can get a much sharper statement for $R \gg \nu^{-1/2}$. Computing the It\^o differential, we have the energy estimate
\[\frac{d}{dt} \E \|\psi^\nu_t\|_{L^2}^2 = - 2\nu \E\|\nabla \psi^\nu_t\|_{L^2}^2 + \|g\|_{L^2}^2.\]
By the It\^o isometry, Duhamel's principle, and the stationarity of the $u_t$ process,
\[\frac{d}{dt} \E \|\psi^\nu_t\|_{L^2}^2  = \E  \|\phi^\nu_t\|_{L^2}^2 \leq e^{-C^{-1} \nu t} \stackrel{t \to \infty}{\to} 0.\]
Taking $t \to \infty$, we have the energy balance
\[
    2\nu \lim_{t \to\infty}\E\|\nabla \psi^\nu_t\|_{L^2}^2 =\|g\|_{L^2}^2 =1.
\]
Thus for any $R \geq 1$,
\[m^\nu([R,\infty)) = \lim_{t \to \infty} \E \|\Pi_{\geq R} \psi^\nu_t\|_{L^2}^2 \leq R^{-2} \lim_{t \to \infty} \E\|\nabla \psi_t^\nu\|_{L^2}^2 \leq \frac{1}{\nu R^2}.\]
This implies for $r \gg \nu^{-1/2}$, effectively all the annuli have very little mass. For example, for $r \geq \nu^{-1}$, not even a single annulus has $\frac{2}{r}$ mass, as for every such $r$,
\[m^\nu([r,r+h]) \leq m^\nu([r,\infty)) \leq \frac{1}{\nu r^2} \leq \frac{1}{r}.\]

\subsection{Overview of the argument and the proof of a special case}
\label{ss:special-case}

The tool for the proof is Theorem~\ref{thm:flux}, which is a pointwise inequality for every $t, r, \nu$, as well as every random event $\omega$. As can be seen from its short proof, Theorem~\ref{thm:flux} is effectively a direct computation of a ``flux'', together with the introduction of a weight before applying Young's convolution inequality.

\subsubsection{\texorpdfstring{The case of an almost sure bound on $u$}{The case of an almost sure bound on u}}

To understand the utility of Theorem~\ref{thm:flux}, we first discuss a special case of the main results when we assume an almost sure bound on $u_t$, that is, we suppose that for some weight function $w$, we have some constant $A_w >0$ such that
\[\|w(|k|) \hat u_t\|_{\ell^1} \leq A_w \qquad \text{almost surely.}\]
In this case, taking an expectation and integrating over $t\in [0,\infty)$, Theorem~\ref{thm:flux} implies
\[4\pi A_w rw^{-1}*m^\nu(r) + \|\Pi_{\geq r} g\|_{L^2}^2 \geq 2\nu \int_0^\infty \E\|\Pi_{\geq r} \nabla \phi_t^\nu\|_{L^2}^2\,\dif t,\]
where we use that $\phi^\nu_t \stackrel{t\to\infty}{\to} 0$.

Next, we specialize to the \textit{Batchelor regime}. The Batchelor regime should be beyond the \textit{injection range} in which the effect of the forcing $g$ is dominant, so we should only consider $r \geq N_0$ so that $\|\Pi_{\geq r} g\|_{L^2}^2\leq \frac{1}{4}$. On the other hand, the Batchelor regime should be below the dissipative regime, where the dissipation is dominant. As such we should only consider $r \leq D_\nu$ where $D_\nu$ satisfies
\begin{equation}
\label{eq:dissipation-scale-def}
2\nu \int_0^\infty \E\|\Pi_{\geq D_\nu} \nabla \phi_t^\nu\|_{L^2}^2\,\dif t = \frac{1}{2}.
\end{equation}
Note that Proposition~\ref{prop:mix-dissipation-bound} gives that, under the exponential mixing assumption, $D_\nu \geq \frac{\nu^{-1/2}}{C (\log \nu^{-1})^{1/2}}$.

In the Batchelor regime, $N_0 \leq r \leq D_\nu$, we have
\begin{equation}
    \label{eq:convolved-mass-bound}
w^{-1}*m^\nu(r)\geq \frac{1}{16 \pi A_w r}.
\end{equation}
We then want $\frac{1}{w}$ to decay rapidly. In the arguments below, we will consider
\[\frac{1}{w(s)} = (1+|s|)^{-p} \text{ or } \frac{1}{w(s)} = \frac{1}{L} \indc_{|s| \leq L}.\]
In either case,~\eqref{eq:convolved-mass-bound} states that the mass measure has order at least $\frac{1}{r}$ \textit{after convolving with some rapidly decaying function.}

In the case that $\frac{1}{w(s)} = \frac{1}{L} \indc_{|s| \leq L}$, the convolution can be directly expanded to give
\[m^\nu([r,r+2L]) \geq \frac{1}{C A_w r},\]
which is then the same bound as is given by~\eqref{eq:compact-as}. Note that we need $\|w(|k|) \hat u_t\|_{\ell^1}$ to be a.s.\ bounded, which requires that $\hat u_t$ has compact support in Fourier space (as $w \equiv +\infty$ on $(L, \infty)$).

In the case $\frac{1}{w(s)} = (1+|s|)^{-p}$, we can rather compute
\[\frac{1}{16 \pi A_w} r^{-1} * \indc_{[h/4,3h/4]} \leq \frac{1}{16\pi A_w}  w^{-1} * m^\nu * \indc_{[h/4,3h/4]} \leq \frac{2}{16\pi A_w}  m^\nu * \indc_{[0,h]} + \rho_{h,w}*m^\nu,\]
where $w^{-1} * \indc_{[h/4,3h/4]} \leq 2 \indc_{[0,h]} + \rho_{h,w}$, and one can show that $\rho_{h,w}$ both decays quickly at infinity and is small in $L^1$ as $h \to \infty$. Then, noting that $m^\nu * \indc_{[0,h]}(r) = m^\nu([r,r+h])$, we have that, for $r$ in the Batchelor regime,
\[m^\nu([r,r+h]) > \frac{h}{CA_w r} \quad \text{or} \quad \rho_{h,w} * m^\nu(r) \geq \frac{1}{CA_w r}.\]
Thus the density where $m^\nu([r,r+h]) \leq \frac{h}{CA_w r}$ is controlled by the density where $\rho_{h,w} * m^\nu(r) \geq \frac{1}{CA_w r}$. To conclude, we compute a bound
\begin{equation}
\label{eq:remainder-bound}
\int_1^R\rho_{h,w} * m^\nu(r) \leq M_w(h) \log R,
\end{equation}
where $M_w(h) \to 0$ as $h \to \infty$. A Chebyshev-like bound allows us to conclude that the logarithmic density of $r$ such that $\rho_{h,w} * m^\nu(r) \geq \frac{1}{CA_w r}$ goes to $0$ as $h \to \infty$, exactly as desired. This then gives (a qualitative version of)~\eqref{eq:general-as}.

In order to get the bound~\eqref{eq:remainder-bound}, we need to use the upper bound on $m^\nu([1,R])$ given by Proposition~\ref{prop:upper-density-bound}, which in turn relies on exponential mixing. As such, in order to get a density bound on the set of undercharged annuli in the case that $u$ has only finite regularity---corresponding to the weight function $\frac{1}{w(s)} = (1+|s|)^{-p}$---we need to use the exponential mixing assumption more strongly. This seems necessary, as otherwise we could have the inequality~\eqref{eq:convolved-mass-bound} hold despite having large regions of Fourier space devoid of mass. In that case, huge spikes of mass would compensate for the decay induced by $w^{-1}$. Exponential mixing allows us to rule out the possibility of too many large spikes.

\subsubsection{The general case}

The argument of the previous section has made our strategy clear: first use an integrated version of Theorem~\ref{thm:flux} to get the right lower bound on the convolved density of the mass measure, then show that this lower bound on the convolved mass density implies that most intervals have the right mass.

In the case that we do not have an almost sure bound on $u_t$---but rather only have some finite moment bound on $u_t$---we have to be more careful. We cannot simply take expectation and integrate Theorem~\ref{thm:flux} over $t \in [0,\infty)$ to obtain a bound on $m^\nu$, as the term $\|w(|k|) \hat u_t\|_{\ell^1}$ no longer comes out of the $\E \int_0^\infty$ integral cleanly. Instead we must integrate over \textit{finite time} and apply H\"older's inequality to the $\E \int_0^T$ integral. We will now pay by a constant growing in time, and as such we want to take $T$ as small as possible. However, in order for Theorem~\ref{thm:flux} to be useful at some fixed frequency $r$ on the time interval $[0,T]$, we need that some nontrivial amount of Fourier mass has been transferred from wavenumbers $|k| < r$ to wavenumbers $|k| \geq r$ on the time interval $[0,T]$. This is ensured for $T \approx \log r$ by the uniform-in-diffusivity exponential mixing.

By localizing to the time interval $[0,C \log r]$ and carefully dealing with the new error introduced by applying H\"older's inequality, we can essentially recover the argument sketched above in the case we only have finite moments on $u_t$. This does however introduce the additional requirement that we take $\alpha \to 0$ as well as $h \to \infty$ to get that $\mu_{N_0, R}(B^\nu_{h,\alpha}) \to 0,$ in contrast to the case of almost sure bounds on $u$, where we could always work at some fixed sufficiently small $\alpha$.

\subsection{Interpretation of the main results and necessity of the hypotheses}

We now briefly discuss why the stated results are reasonably natural. First we consider~\eqref{eq:compact-as}. This result says that for a.s.\ bounded and compactly supported in Fourier space advecting flows \textit{every} annulus of sufficiently large width have the predicted amount of Fourier mass. This is reasonable, as the compactness of Fourier support of the advecting flow prevents mass of the passive scalar from jumping too far in Fourier space, as is clear from~\eqref{eq:fourier-evolution}, while the almost sure bound prevents mass from moving in Fourier space too fast so as to undercharge an annulus in Fourier space despite the mass moving through the annulus.

Next we consider~\eqref{eq:general-as}. In this case, we need to take $h \to \infty$ to ensure that the density of Fourier annuli which have the predicted amount of mass tends to $1$. Note, however, that we do not need to change $\alpha$. This is sensible as the a.s.\ bound on $u$ ensures that mass cannot move too fast in Fourier space, which is what should govern the choice of $\alpha$, but the fact that $u$ no longer has compact Fourier support allows large Fourier modes in $u$ to allow mass to jump over some annuli, again as is suggested by~\eqref{eq:fourier-evolution}. However, it cannot be that the Fourier mass is \textit{typically} using these large Fourier modes of $u$ to jump over annuli in Fourier space, since---due to the regularity of $u$---the large Fourier modes of $u$ have very small coefficients, and so if the Fourier mass were \textit{typically} traveling using these large Fourier modes, it would move to $\infty$ \textit{too slowly to be compatible with exponential mixing}.

That is, exponential mixing prescribes a rate at which the Fourier mass must move to $\infty$ and the regularity of $u$ says that Fourier mass cannot possibly move to $\infty$ fast enough if the mass moves purely along the large Fourier modes of $u$. Thus, the mass must \textit{typically} travel along the small Fourier modes, which causes typical annuli with sufficiently large width to be charged with the predicted amount of Fourier mass.

Now we consider~\eqref{eq:compact-general}. In this case, we have compact support in Fourier space---and as such only ever need to consider annuli of a fixed sufficiently large width---but we do not have an a.s.\ bound on $u$. As such, there is no a.s.\ upper bound on how fast mass can move across Fourier modes. It is possible for the advection to move mass through a given annulus very rapidly, and thus only charge the annulus with $\frac{\alpha h}{r}$ mass for some very small $\alpha$. The moment bounds on $u$ ensure that $u$ is \textit{typically} not so large, and so \textit{typically} the annuli are not severely undercharged with mass. However, we see now that it is necessary to take $\alpha$ small in~\eqref{eq:compact-general} if we want the undercharged annuli to be sparse.

Finally we consider~\eqref{eq:general-general}. In this setting, the problems of~\eqref{eq:general-as} and~\eqref{eq:compact-general} are present. While \textit{typically} Fourier mass moves at a controlled rate via small Fourier modes in $u$, it can atypically move arbitrarily fast or along large Fourier modes of $u$. To show that the set of undercharged annuli are sparse, we must therefore simultaneously take $h \to \infty$ and $\alpha \to 0$.

One may observe that our argument uses the time-stationarity of $t \mapsto u_t$ only once, when relating the asymptotic (deterministic) and instantaneous (random) mass measures. We hope that a more careful analysis using stationarity can prove a similar result for \textit{every} annulus, not just a high logarithmic density of annuli, but this is out of scope of the current paper.

\section{Proof of the main results}

We start by precisely defining the Fourier projector notation used above.

\begin{definition}
    \label{def:fourier-projector}
    For a radius $r \geq 0$, we define the complementary $L^2$ orthogonal Fourier projections $\Pi_{\geq r}, \Pi_{<r} : L^2(\T^d) \to L^2(\T^d)$,
\begin{equation}
\label{eq:Fourier-projectors-def}
\Pi_{\geq r} f := \sum_{|k| \geq r} e^{2\pi i k \cdot x} \hat f(k) \quad \text{and} \quad\Pi_{< r} f := \sum_{|k| < r} e^{2\pi i k \cdot x} \hat f(k) =  f- \Pi_{\geq r} f.
\end{equation}
\end{definition}

We now Theorem~\ref{thm:flux} by direct computation.

\begin{proof}[Proof of Theorem~\ref{thm:flux}]

Using the equation~\eqref{eq:phi-equation}, we have
\begin{equation}
\label{eq:fourier-evolution}
\frac{\dif}{\dif t} |\hat \phi^\nu_t|^2(k) = -8 \pi^2 |k|^2 \nu |\hat \phi^\nu_t|^2(k) +2\pi i k \cdot \sum_j   \hat \phi^\nu_t(k)\overline{ \hat u_t(k-j) \hat \phi_t^\nu(j)}-\overline{\hat \phi^\nu_t(k)} \hat u_t(k-j) \hat \phi_t^\nu(j).
\end{equation}
We then have
\begin{align*}
\frac{d}{dt} \|\Pi_{\geq r} \phi_t^\nu\|_{L^2}^2 &=-2\nu\|\Pi_{\geq r} \nabla \phi_t^\nu\|_{L^2}^2
    +\sum_{|k| \geq r} 2\pi i k \cdot \sum_j   \hat \phi^\nu_t(k)\overline{ \hat u_t(k-j) \hat \phi_t^\nu(j)}-\overline{\hat \phi^\nu_t(k)} \hat u_t(k-j) \hat \phi_t^\nu(j)
    \\&= -2\nu\|\Pi_{\geq r} \nabla \phi_t^\nu\|_{L^2}^2  +  \sum_{|k| \geq r} 2\pi i j \cdot \sum_{|j| < r}   \hat \phi^\nu_t(k)\overline{ \hat u_t(k-j) \hat \phi_t^\nu(j)}-\overline{\hat \phi^\nu_t(k)} \hat u_t(k-j) \hat \phi_t^\nu(j)
     \\&\leq -2\nu \|\Pi_{\geq r} \nabla \phi_t^\nu\|_{L^2}^2  +  4\pi r \sum_{|k| \geq r}   \sum_{|j| < r}   |\hat \phi^\nu_t(k)|| \hat u_t(k-j)| |\hat \phi_t^\nu(j)|
\end{align*}
where we use that $k \cdot \hat u_t(k-j) = j \cdot \hat u_t(k-j)$ as $\nabla \cdot u_t =0$. Then for an arbitrary increasing weight $w : [0,\infty) \to [1,\infty],$
\begin{align*}
     \sum_{|k| \geq r}   \sum_{|j| < r}   &|\hat \phi^\nu_t(k)|| \hat u_t(k-j)| |\hat \phi_t^\nu(j)|
     \\&\leq  \sum_k   \sum_j  w(||k| - r|)^{-1/2}|\hat \phi^\nu_t(k)| w(|k-j|)|\hat u_t(k-j)| w(||j| - r|)^{-1/2} |\hat \phi_t^\nu(j)|
     \\&\leq \| w(|k|) \hat u_t\|_{\ell^1} \sum_k w(||k|-r|)^{-1} |\hat \phi_t^\nu(k)|^2 = \|w(|k|) \hat u_t\|_{\ell^1} w^{-1} * m^\nu_t(r),
\end{align*}
where we use Young's convolution inequality for the second inequality. Combining the inequalities, we conclude.
\end{proof}

This next proposition provides a lower bound for the left-hand side of Theorem~\ref{thm:flux} when integrated over the time interval $[0,C \log r]$. This allows us to localize the time on which the flux occurs, which is needed to apply H\"older's inequality to the $\E \int_0^T$ integral when bounding the $\|w(|k|) \hat u_t\|_{\ell^1}$ term.

\begin{proposition}
\label{prop:mix-dissipation-bound}
    For all
    \[2  \leq r \leq \frac{\nu^{-1/2}}{C (\log \nu^{-1})^{1/2}} \quad \text{and}\quad T \geq C\log r,\] and we have that
    \[\E \|\Pi_{\geq r} \phi_{T}^\nu\|_{L^2}^2 + 2\nu \int_0^{T} \E\|\Pi_{\geq r} \nabla \phi_t^\nu\|_{L^2}^2\,\dif t \geq \frac{1}{2}\]
\end{proposition}

\begin{proof}
    By the energy identity and the initial condition $\|\phi_0^\nu\|_{L^2}^2 = \|g\|_{L^2}^2 = 1$, we have
    \[ \|\phi_{T}^\nu\|_{L^2}^2  + 2\nu \int_0^{T}  \|\nabla \phi_t^\nu\|_{L^2}^2\,\dif t = 1.\]
    Thus
    \[\E \|\Pi_{\geq r} \phi_{T}^\nu\|_{L^2}^2 + 2\nu \int_0^{T} \E\|\Pi_{\geq r} \nabla \phi_t^\nu\|_{L^2}^2\,\dif t = 1 - \E \|\Pi_{< r} \phi_{T}^\nu\|_{L^2}^2 - 2\nu \int_0^{T} \E\|\Pi_{< r} \nabla \phi_t^\nu\|_{L^2}^2\,\dif t\]
    By the exponential mixing, we have that
    \[ \E \|\Pi_{< r} \phi_{T}^\nu\|_{L^2}^2  \leq  (2\pi r)^2\E \|\phi_{T}^\nu\|_{H^{-1}}^2 \leq (2 \pi r)^2K  e^{-\gamma T} \leq (2 \pi r)^2e^{-\gamma C \log r} \leq \frac{1}{4},\]
    where we choose $C$ large enough. For the last term, we compute, for any $\alpha >0$,
    \begin{align*}
    2\nu \int_0^T \E\|\Pi_{< r} \nabla \phi_t^\nu\|_{L^2}^2\,\dif t
    &\leq (2 \pi r)^2 \nu \int_0^{C\log r} \E \|\phi_t^\nu\|_{L^2}^2\,\dif t + (2 \pi r)^4 \int_{C \log r}^\infty \E\|\phi_t^\nu\|_{H^{-1}}^2\dif t
    \\&\leq  (2 \pi r)^2 \nu C \log r + (2 \pi r)^4  \int_{C \log r }^\infty Ke^{-\gamma t} \dif t
    \\&\leq   C \nu (2 \pi r)^2 \log r + (2 \pi r)^4K e^{- \gamma C \log r}
    \\&\leq \frac{1}{4},
\end{align*}
where the last line follows by taking $C$ large enough.
\end{proof}

\begin{corollary}
    For all
    \[C  \leq r \leq \frac{\nu^{-1/2}}{C (\log \nu^{-1})^{1/2}} \quad \text{and}\quad  T \geq C \log r,\]
    we have
    \begin{equation}
    \label{eq:pointwise-flux-inequality}
    \E \int_0^T \|w(|k|) \hat u_t\|_{\ell^1} w^{-1}*m^\nu_t(r)\,\dif t  \geq \frac{1}{Cr}.
    \end{equation}
\end{corollary}

\begin{proof}
    Using Theorem~\ref{thm:flux} and then Proposition~\ref{prop:mix-dissipation-bound},
    \begin{align*}
    4\pi \E\int_0^T \|w(|k|) \hat u_t\|_{\ell^1} r w^{-1}*m^\nu_t(r)\,\dif t &\geq  \E \|\Pi_{\geq r} \phi_{T}^\nu\|_{L^2}^2 + 2\nu \int_0^{T} \E\|\Pi_{\geq r} \nabla \phi_t^\nu\|_{L^2}^2\,\dif t - \|\Pi_{\geq r} g\|_{L^2}^2
    \\&\geq \frac{1}{2} - \|\Pi_{\geq r} g\|_{L^2}^2 \geq \frac{1}{4},
    \end{align*}
    where we take $N_0 \leq C \leq r$ so that $\|\Pi_{\geq r} g\|_{L^2}^2 \leq \frac{1}{4}$. Rearranging, we conclude.
\end{proof}

For the arbitrary increasing weight $w$, we define a set analogous to $B^\nu_{h,\alpha}$ as follows
\begin{equation}
\label{eq:bad-radii-weighted}
    B^\nu_{w,\alpha} := \big\{r \in [1,\infty) : w^{-1} * m^\nu(r) \leq \frac{\alpha}{r}\big\}.
\end{equation}
$B^\nu_{w,\alpha}$ is then the radii for which the convolved density $w^{-1} * m^\nu(r)$ is undercharged with mass with lower bound $\alpha$. The next result shows that the log-density of this set is well controlled as $\alpha \to 0$.

\begin{proposition}
\label{prop:B-w-bound}
    Suppose
    \[\int_0^\infty \frac{1}{w(s)}\,\dif s \leq \frac{1}{2}.\]
    Then there is a constant $C > 0$ such that, for any
    \[C\leq R \leq \frac{\nu^{-1/2}}{C (\log \nu^{-1})^{1/2}},\]
    and any $p\geq 1,$ we have that
    \[\mu_{N_0, R}(B^\nu_{w,\alpha}) \leq  C^p  \E \|w(|k|) \hat u_0\|_{\ell^1}^p\alpha^{p-1}.\]
\end{proposition}

\begin{proof}
    Note first that
    \begin{equation*}
      \int w^{-1} * m_{t,\omega}^\nu(r)\,\dif r \leq 2\|w^{-1}\|_{L^1} \int m_{t,\omega}^\nu(r)\,\dif r \leq \sum_k |\hat \phi^\nu_t|^2(k) \leq 1,
    \end{equation*}
    so for $q \geq 1$ we have
    \begin{equation*}
      \left(\int w^{-1} * m_{t,\omega}^\nu(r)\,\dif r\right)^{q} \leq \int w^{-1} * m_{t,\omega}^\nu(r)\,\dif r.
    \end{equation*}
    Integrating~\eqref{eq:pointwise-flux-inequality} over $E:= B^\nu_{w,\alpha} \cap [N_0,R]$, we get that
    \begin{align*} \int_E \frac{1}{r}\,\dif r &\leq  C\int_0^{C \log R}  \E\|w(|k|) \hat u_t\|_{\ell^1} \int_E  w^{-1}*m^\nu_t(r)\,\dif r\dif t
      \\&\leq   C\Big(\int_0^{C \log R}\E\|w(|k|) \hat u_t\|_{\ell^1}^p\,\dif t \Big)^{1/p} \Big(\E \int_0^{C \log R}\Big(\int_E  w^{-1}*m^\nu_t(r)\,\dif r\Big)^q\dif t\Big)^{1/q}
      \\&\leq  C\big(C\E \|w(|k|) \hat u_0\|_{\ell^1}^p \log R\big)^{1/p}  \Big(\E \int_0^\infty \int_E  w^{-1}*m^\nu_t(r)\,\dif r\dif t\Big)^{1/q}
      \\&=   C\big(C\E \|w(|k|) \hat u_0\|_{\ell^1}^p \log R\big)^{1/p}   \Big(\int_E  w^{-1}*m^\nu(r)\,\dif r\Big)^{1/q}
      \\&\leq C\big(C\E \|w(|k|) \hat u_0\|_{\ell^1}^p \log R\big)^{1/p}  \alpha^{1/q}  \Big(\int_E  \frac{1}{r}\,\dif r\Big)^{1/q}
    \end{align*}
    where $p^{-1} + q^{-1} = 1$ and we use that $E \subseteq B^\nu_{w,\alpha}$ for the final line. Rearranging and using the definition of $\mu_{N_0,R},$ we conclude.
\end{proof}

In the following arguments, we prove~\eqref{eq:general-general} and~\eqref{eq:compact-general}. From here,~\eqref{eq:general-as} and~\eqref{eq:compact-as} follow by choosing $C$ large enough and sending $p \to \infty$.

We first show~\eqref{eq:compact-general}, which follows from applying Proposition~\ref{prop:B-w-bound} with $w^{-1}(s) := (2L)^{-1} \indc_{s \leq L}$.

\begin{proof}[Proof of Theorem~\ref{thm:compact-fourier-support}]
    We choose
    \[w(s) := 2L\indc_{s \leq L} + \infty \indc_{s > L},\]
    and note that $\int_0^\infty \frac{1}{w(s)}\,\dif s \leq 1/2$ and
    \[\E \|w(|k|)\hat u\|_{\ell^1}^p \leq \E \Big(\sum_{|k| \leq L} L X\Big)^p \leq C^p L^{(d+1)p} \E X^p.\]
    We also have that
    \[w^{-1}* m^\nu(r) = \frac{1}{2L} m^\nu([r-L,r+L]),\]
    where the convolution is understood in the sense of~\eqref{eq:conv-def}.
    Thus
    \[B^\nu_{2L,\alpha} = B^\nu_{w,\alpha} +L.\]
    We conclude by Proposition~\ref{prop:B-w-bound}.
\end{proof}

Finally we prove~\eqref{eq:general-general}, which follows from first applying Proposition~\ref{prop:B-w-bound} with $w_q(s) := (1+ |s|)^q$ and using the lower bound on $w_q^{-1} * m^\nu$ to lower bound $m^\nu([r,r+h])$.

\begin{proof}[Proof of Theorem~\ref{thm:general}]
    Fix $q \geq 2$ and let
    \[w_q(r) := 2(1+r)^{q}.\]
    Then note that $\int_{0}^\infty \frac{1}{w_q(s)}\,\dif s \leq 1/2$ and
    \begin{equation}
      \label{eq:only-use-of-Cq}
      \|w_q(|k|) \hat u_0\|_{\ell^1} \leq  C^q\| |k|^q \hat u_0\|_{\ell^1} \leq C^q \||k|^{q+(d+1)/2} \hat u_0\|_{\ell^2} \leq C^q \|u_0\|_{H^{q+(d+1)/2}} \leq C^q \|u_0\|_{C^{q+(d+1)/2}}.
    \end{equation}
    For $r\geq h,$ for any $\beta>0$,
    \begin{align*}
        \frac{h\beta}{4r} &\leq \int_{r+h/4}^{r+3h/4} \frac{\beta}{s}\,\dif s
        \\&\leq  \int_{r+h/4}^{r+3h/4}w_q^{-1}* m^\nu(s)\,\dif s + \int_{[r+h/4,r+3h/4] \cap B^\nu_{w_q,\beta}}\frac{\beta}{s}\,\dif s
        \\&\leq \int_{r+h/4}^{r+3h/4} \int w_q^{-1}(|s-z|) m^\nu(\dif z)\dif s + \int_{[r,r+h] \cap B^\nu_{w_q,\beta}}\frac{\beta}{s}\,\dif s
        \\&\leq 2 m^\nu([r,r+h]) + \int_{r+h/4}^{r+3h/4}\int_{[1,\infty)\backslash[r,r+h]}  w_q^{-1}(|s-z|) m^\nu(\dif z)\dif s + \int_{[r,r+h] \cap B^\nu_{w_q,\beta}}\frac{\beta}{s}\,\dif s
        \\&=: 2m^\nu([r,r+h]) + R_{q,h,\beta}(r).
    \end{align*}
    Our goal now is to show that $\int_{N_0}^R R_{q,h,\beta}(r)\,\dif r$ is suitably small, which will in turn imply---by essentially Chebyshev's inequality---that $R_{q,h,\beta}(r)$ is typically small, and so that $m^\nu([r,r+h])$ typically has the right lower bound.

    Integrating the first term, we see that
    \begin{align*}
\int_1^R\int_{r+h/4}^{r+3h/4}&\int_{[1,\infty)\backslash[r,r+h]}  w_q^{-1}(|s-z|) m^\nu(\dif z)\dif s\dif r
\\&\leq h\int_1^{R+h}\int_1^\infty  \indc_{|s-z|> h/4}w_q^{-1}(|s-z|) m^\nu(\dif z)\dif s
\\&\leq h \int_1^{2R} \int_1^{3R} \indc_{|s-z|> h/4}w_q^{-1}(|s-z|) m^\nu(\dif z)\dif s + h \int_1^{2R} \int_{3R}^\infty w_q^{-1}(|s-z|) m^\nu(\dif z)\dif s
\\&\leq2h m^\nu([1,3R])\int_{h/4}^\infty s^{-q}\,\dif s + C^q h \int_{3R}^\infty z^{1-q} m^\nu(\dif z)
\\&\leq C^q h^{2-q}m^\nu([1,3R])+ C^q h  \int_{3R}^\infty z^{1-q} \frac{\dif}{\dif z} m^\nu([1,z]) \dif z
\\&\leq C^q h^{2-q} \log(R)+ C^q h \int_{3R}^\infty z^{-q}\log(z) \dif z
\\&\leq C^q h^{2-q} \log(R)+ C^q h R^{2-q}
\\&\leq C^q h^{3-q} \log(R),
    \end{align*}
where we use~\eqref{eq:cumulative-upper-bound} to control $m^\nu([1,a])$ and we use that $R \geq h$. We now use Proposition~\ref{prop:B-w-bound} to integrate the second term of $R_{q,h,\beta}$:
 \begin{align*}
        \int_{N_0}^R\int_{[r,r+h] \cap B^\nu_{w_q,\beta}}\frac{\beta}{s}\,\dif s\dif r &\leq h \int_{B^\nu_{w_q,\beta} \cap [N_0, R+h]} \frac{\beta}{s}\,\dif s
        \\&\leq C  \beta\mu_{[N_0, 2R]}(B^\nu_{w_q,\beta}) \log(R) h
        \\&\leq   C^p  \E \|w_q(|k|) \hat u_0\|_{\ell^1}^p \beta^p \log(R) h
        \\&\leq C^{p+q} \E \|u_0\|_{C^{q+(d+1)/2}}^p \beta^p \log(R)h.
    \end{align*}
    Letting $\alpha := \beta/16$ and combining the displays, we conclude that
    \begin{align*}\mu_{N_0, R}(B^\nu_{h,\alpha}) &\leq \mu_{N_0, R}\{r \in [N_0, R] : R_{q,h, 16\alpha}(r) \geq \frac{h\alpha}{r}\}
    \\&\leq \frac{C}{\log R} \int_{[N_0,R] \cap \{R_{q,h,16\alpha} \geq \frac{h\alpha}{r}\}} \frac{1}{r}\,\dif r
    \\&\leq \frac{C}{\alpha h\log R } \int_{[N_0, R]} R_{q,h,16\alpha}(r)\,\dif r
    \\&\leq C^q h^{2 - q} \alpha^{-1} + C^{p+q} \E \|u_0\|_{C^{q+(d+1)/2}}^p \alpha^{p-1},
    \end{align*}
    as claimed.
\end{proof}

{\small
\bibliographystyle{alpha}
\bibliography{reference-bill,references2}
}

\end{document}